\documentclass[14pt]{amsart}
\usepackage{geometry}                
\usepackage{graphicx}
\usepackage{amssymb}
\usepackage{amsmath}
\usepackage{epstopdf}
\usepackage{xcolor}
\usepackage{amsmath,bm}
\numberwithin{equation}{section}

\newtheorem{thm}{Theorem}[section]
\newtheorem{lem}[thm]{Lemma}

\newtheorem{hypothesis}{Hypothesis \hspace{-1.2mm}}

\DeclareMathOperator{\grad}{grad}
\DeclareMathOperator{\ddiv}{div}
\DeclareMathOperator{\rot}{rot}
\DeclareMathOperator{\Curl}{curl}
\DeclareMathOperator{\tr}{tr}

\DeclareMathOperator{\osc}{osc}

\def\grad{\operatorname{grad}}
\DeclareMathOperator{\A}{\mathcal{A}}

\title[Quasi-optimal convergence for AMFEMs]{A unified analysis of quasi-optimal convergence for adaptive mixed finite element methods}

\author[J. Hu]{Jun Hu$^{1}$}
\author[G. Yu]{Guozhu Yu$^{2,*}$}\thanks{*: Corresponding author.}
  \dedicatory{$^1$LMAM and School of Mathematical Sciences, Peking University, Beijing 100871,China\\
              $^2$School of Mathematics, Southwest Jiaotong University, Chengdu  610031, China}
  \thanks{The research of the first author was supported by NSFC projection 11271035,  91430213 and 11421101. }
  \thanks{The research of the second author was supported by NSFC projection 11401492.}
  \thanks{Email addresses: hujun@math.pku.edu.cn (J. Hu), yuguozhumail@163.com (G. Yu)}

\begin{document}
\maketitle
\begin{abstract}
In this paper, we  present a unified analysis of  both convergence and optimality of adaptive mixed finite element
methods  for a class of problems when the finite element spaces and corresponding a posteriori error estimates under consideration satisfy  five hypotheses.
We prove that these five conditions are sufficient for  convergence and optimality of the adaptive algorithms under consideration.
The main ingredient for the  analysis is a new method to analyze both discrete reliability  and  quasi-orthogonality. This new method arises from
an appropriate and natural choice of the norms for both  the discrete  displacement and stress spaces, namely, a mesh-dependent discrete $H^1$ norm
for the former and a $L^2$ norm for the latter, and  a newly defined projection  operator from the discrete stress space on the coarser mesh onto the discrete divergence free space on the finer mesh.  As  applications, we prove these five hypotheses for the Raviart--Thomas and Brezzi--Douglas--Marini elements of  the Poisson  and  Stokes problems in both 2D and 3D.
\end{abstract}


\section{Introduction}
This paper is devoted to convergence and optimality of adaptive mixed finite element methods (AMFEMs) for the problem of the following form:
Given $f\in L^2(\Omega)$, find $(\sigma, u)\in \Sigma\times U$ such that
\begin{equation}\label{eq:UnifiedProblem}
\begin{split}
(\A\sigma, \tau)_{L^2(\Omega)}-(\ddiv\tau, u)_{L^2(\Omega)}=0,\text{ for any }\tau\in \Sigma,\\[0,5ex]
(\ddiv\sigma,v)_{L^2(\Omega)}-(f,v)_{L^2(\Omega)}=0, \text{ for any }v\in U.
\end{split}
\end{equation}
In the paper, we refer to $\Sigma$ as the stress space, and $U$ as the displacement space.
Here, $\Omega$ is a simply connected bounded domain in $\mathbb{R}^d (d=2,3)$ with the boundary $\partial\Omega$,
and $\Sigma$, $U$ are Sobolev spaces defined as
$$\Sigma:=H(\ddiv,\Omega;\mathbb{R}^{d\times n}),\quad U:=L^2(\Omega;\mathbb{R}^n)$$
with $n$ some positive integer. Furthermore, we assume $\A$ is a linear, bounded and semi-definite operator, satisfying
\begin{equation}\label{coe}
0\le(\A\tau, \tau)_{L^2(\Omega)}\le C\|\tau\|_{L^2(\Omega)}^2 \mbox{~and~} \|\tau\|_{L^2(\Omega)}^2\le C((\A\tau, \tau)_{L^2(\Omega)}+\|\ddiv\tau\|^2_{H^{-1}(\Omega)})
\end{equation}
for any $\tau\in\Sigma$, with the $H^{-1}(\Omega)$ norm defined as
$$\|\psi\|_{H^{-1}(\Omega)}:=\sup_{v\in H_0^1(\Omega)}\frac{(\psi,v)}{\|v\|_{H^1(\Omega)}}.$$
For convenience, we will also use $\|\tau\|_{\A}$ to denote $(\A\tau, \tau)_{L^2(\Omega)}^{1/2}$ when there is no confusion.

Many problems can be attributed to the form of \eqref{eq:UnifiedProblem}. For example, when
$$\A\tau:=\tau,\quad \Sigma\times U:=H(\ddiv,\Omega;\mathbb{R}^d)\times L^2(\Omega;\mathbb{R}),$$
the problem \eqref{eq:UnifiedProblem} is essentially the mixed formulation of the Poisson problem; when
$$\A\tau:=\tau-\frac1d(\tr\tau) I_{d\times d},\quad \Sigma\times 
U:=\left\{\tau\in H(\ddiv,\Omega;\mathbb{R}^{d\times d})|\int_\Omega\tr \tau dx=0\right\}\times L^2(\Omega;\mathbb{R}^d)
$$
with the $d\times d$ identity matrix $I_{d\times d}$ , then the problem \eqref{eq:UnifiedProblem} becomes  the pseudostress-velocity formulation of the stationary Stokes problem, see for instance,  \cite{arnold1988new, cai2004least, cai2010mixed, carstensen2011priori} and the references therein.
Here and  throughout this paper, the trace operator $\tr$ is  defined as
$$
\tr \tau=\sum\limits_{i=1}^d\tau_{ii}~~ \text{ for any matrix }\tau\in \mathbb{R}^{d\times d}.
$$

For the problem \eqref{eq:UnifiedProblem},  the theory of the reliable and efficient a posteriori error analysis has been in some sense relatively mature.
We  refer the interested readers to \cite{ brandts1994superconvergence, alonso1996error, braess1996posteriori, carstensen1997posteriori,wohlmuth1999comparison,wheeler2005posteriori, lovadina2006energy,ainsworth2007posteriori,larson2008posteriori, kim2012guaranteed} and the references therein for a posteriori error estimates of the mixed finite element methods of the Poisson problem and \cite{carstensen2011priori} for a posteriori error estimates of the mixed finite element methods of the Stokes problem within the pseudostress-velocity formulation. We also mention the references \cite{ hoppe1997adaptive, carstensen2005unifying, vohralik2007posteriori,carstensen2012review, du2014residual} for the other  related works.

As for the convergence and optimality analysis, there have been several results for adaptive conforming and nonconforming finite element methods \cite{babuvska1984feedback, dorfler1996convergent, morin2000data,morin2003local,binev2004adaptive,mekchay2005convergence,hu2007convergence, stevenson2007optimality, cascon2008quasi,morin2008basic,hu2012convergence,carstensen2014axioms};
while for AMFEMs, research efforts are made mainly on the Poisson problem.
 Carstensen and Hoppe \cite{carstensen2006error} established the first error reduction and convergence  of  the adaptive lowest-order Raviart--Thomas  element method, and  similar results can be found  in \cite{becker2008optimally, carstensen2011optimal}.
Later, in \cite{chen2009convergence, du2012error} convergence and optimality were analyzed for the Raviart--Thomas and Brezzi--Douglas--Marini elements of any order.
By using the discrete Helmholtz decomposition, Huang and Xu \cite{huang2012convergence}
extended the above results to the 3D case.  For the mixed finite elements of the Stokes problem within  the
pseudostress-velocity formulation, Carstensen et. al \cite{carstensen2013quasi} proved
convergence and optimality of the adaptive lowest-order Raviart--Thomas element.  The main ingredients therein are some novel equivalence between
  the lowest-order Raviart--Thomas and   Crouzeix--Raviart elements, and  a particular Helmholtz decomposition of deviatoric tensors for the 2D case. However, the analysis can neither be generalized  to the Raviart--Thomas and Brezzi--Douglas--Marini elements of any order, nor to the 3D case.  We refer interested  readers to \cite{carstensen2014axioms} for a comprehensive review of the state of art of this field
  and also a wonderful simultaneous axiomatic analysis of both  convergence and optimality
  of adaptive finite element methods of several classes of linear and nonlinear problems.

This paper aims at a unified convergence and optimality analysis of  AMFEMs  of  the problem \eqref{eq:UnifiedProblem}
in both two and three dimensional cases.  The main result states  that if the mixed finite element methods and  associated a posteriori error estimates
 satisfy five hypotheses, see more details in next section,  the corresponding adaptive algorithms converge with optimal rates in the nonlinear
 approximate  sense. The unified analysis is based on a new method to establish  both discrete reliability  and  quasi-orthogonality, which are two main and indispensable ingredients for the convergence and optimality analysis  of  adaptive finite element methods.  In fact, in contrary to \cite{carstensen2006error,becker2008optimally, chen2009convergence,carstensen2011optimal, du2012error,huang2012convergence,carstensen2013quasi}, the discrete  displacement space  is endowed with a mesh-dependent discrete $H^1$ norm, which defines one component of the new method.  Hence the $L^2$ norm becomes a
    natural norm for the  discrete stress space. The other component of the new method is  to introduce a projection  operator from the discrete stress space on the coarser mesh onto the discrete divergence free space on the finer mesh.
    As  applications,  the Raviart--Thomas and Brezzi--Douglas--Marini elements of any order of both the Poisson problem and  the Stokes problem within  the pseudostress-velocity formulation in 2D and 3D  are proved to satisfy  these five hypotheses.  Therefore  the corresponding adaptive schemes
      admit  optimal convergence.  As a result, it  extends  the optimal convergence result for the first order Raviart--Thomas element of the Stokes
       problem in 2D  from \cite{carstensen2013quasi} to the more general case.

Throughout this paper, the notation $a\lesssim b$ represents that there exists a generic positive constant $C$, which is
independent of the mesh parameter $h$ and may not be the same at different occurrences,
such that $a \le Cb$. The symbol $a \thickapprox  b$ means $a \lesssim b\lesssim a$.

Let $v$, $\beta=(\beta_1,\cdots,\beta_d)^T$ and $\tau=(\tau_{ij})_{d\times d}$ be scalar, vector and tensor functions of two or three variables respectively, and let
$\tau_i=(\tau_{i1},\cdots,\tau_{id})^T$ denote the $i$th row for $\tau$ with $i=1,\cdots,d$. We define the grad, div, curl and rot operators by
$$\grad v:=\left(\frac{\partial v}{\partial x_1},\cdots,\frac{\partial v}{\partial x_d}\right)^T,\quad
\grad \beta:=(\grad\beta_1,\cdots,\grad\beta_d)^T,$$
$$ \ddiv \beta:=\frac{\partial \beta_1}{\partial x_1}+\cdots+\frac{\partial \beta_d}{\partial x_d},\quad
\ddiv \tau:=(\ddiv \tau_1, \cdots, \ddiv \tau_d)^T,$$
$$\Curl v:=\left(\frac{\partial v}{\partial x_2},-\frac{\partial v}{\partial x_1}\right)^T, \quad
\Curl \beta:=(\Curl\beta_1, \Curl\beta_2)^T,\quad d=2,$$
$$\Curl \beta:=\grad\times\beta,\quad
\Curl \tau:=(\Curl\tau_1,\Curl\tau_2,\Curl\tau_3)^T,\quad d=3,$$
 $$\rot\beta:=\frac{\partial \beta_1}{\partial x_2}-\frac{\partial \beta_2}{\partial x_1},\quad
   \rot\tau:=(\rot\tau_1,\rot\tau_2)^T,\quad d=2.$$
Moreover, for the tensor function $\tau$, we define its tangential component by
$$\tau\cdot t:=(\tau_1\cdot t,\tau_2\cdot t)^T\mbox{~for~} d=2, \quad
\tau\times \nu:=(\tau_1\times \nu,\tau_2\times \nu,\tau_3\times \nu)^T\mbox{~for~} d=3.$$
For a given Lebesgue measurable set $G\subset \mathbb{R}^d$, we use $L^2(G;\mathbb{R})$ or $L^2(G;\mathbb{R}^{d\times n})$
to denote the Hilbert space of square integrable functions or matrix-value fields, respectively, with inner product $(\cdot,\cdot)_{L^2(G)}$.
 Here and thereafter we will omit $\mathbb{R}$ or $\mathbb{R}^{d\times n}$ for simplicity when there is no risk of confusion.

We also define the following spaces
$$H^1(G):=\{v\in L^2(G)| \grad v\in L^2(G)\},$$
$$H(\ddiv,G):=\{v\in L^2(G)| \ddiv v\in L^2(G)\},$$
$$H(\Curl,G):=\{v\in L^2(G)| \Curl v\in L^2(G)\},$$
equipped with norms
$$\|v\|_{H^1(G)}:=(\|v\|^2_{L^2(G)}+\|\grad v\|^2_{L^2(G)})^{1/2},\quad\mbox{for all } v\in H^1(G),$$
$$\|v\|_{H(\ddiv,G)}:=(\|v\|^2_{L^2(G)}+\|\ddiv v\|^2_{L^2(G)})^{1/2},\quad\mbox{for all } v\in H(\ddiv,G),$$
$$\|v\|_{H(\Curl,G)}:=(\|v\|^2_{L^2(G)}+\|\Curl v\|^2_{L^2(G)})^{1/2},\quad\mbox{for all } v\in H(\Curl,G),$$
respectively, where $\|\cdot\|_{L^2(G)}:=(\cdot,\cdot)_{L^2(G)}^{1/2}$ denotes the norm of space $L^2(G)$.
Especially, let  $H_0^1(G):=\{v\in H^1(G),v|_{\partial G}=0\}$.

The rest of the paper is organized as follows. In  Section  2, we present notation and five hypotheses. In Section 3, we show discrete reliability
and quasi-orthogonality under these five hypotheses.  In Section 4, we prove convergence and optimality of the adaptive algorithms while in Section 5
we  check these  five hypotheses  for two examples.

\section{Notation and Hypothesis}
Let $\mathcal{T}_h$ be some shape-regular triangulation of $\Omega$
and $\mathcal{E}_h$ the set of all edges or faces in $\mathcal{T}_h$.
We indicate by $h_K:=|K|^{1/d}$ and $h_E:=|E|^{1/(d-1)}$ the size for each $K\in\mathcal{T}_h$ and each $E\in\mathcal{E}_h$, respectively.
Note that all geometric entities are closed sets.  Given $K\in \mathcal{T}_h$ and $E\in \mathcal{E}_h$, define the element-patch and the edge/face-patch by
$$\Omega_K:=\bigcup\limits_{K'\cap K\ne\emptyset}K', \quad\Omega_E:=\bigcup\limits_{E\subset K}K, $$
 respectively.
Given any interior edge/face $E\in\mathcal{E}_h$,  let $\nu_E$ be  a unit normal vector,
and $[\cdot]|_E:=\cdot|_{K^+}-\cdot|_{K^-}$ be the jump across the edge/face $E=K^+\cap K^-$ shared by the two elements $K^+,K^-\in\mathcal{T}_h$.
While for the boundary edge/face $E$, $\nu_E$ denotes the unit outer vector normal to $\partial \Omega$, and the jump
$[\cdot]|_E:=\cdot|_{K^+}$ for the unique element $K^+$ with $E\subset K^+$. When $d=2$
let  $t_E$ be  the  unit tangential vector of $E$.

Let $\mathcal{T}_h$ be a refinement of $\mathcal{T}_H$, and
$\mathcal{R}:=\mathcal{T}_H\backslash\mathcal{T}_h=\{K\in\mathcal{T}_H|K\not\in\mathcal{T}_h\}$ be the set of refined elements from $\mathcal{T}_H$ to $\mathcal{T}_h,$ $\tilde{\mathcal{R}}:=\{K\in\mathcal{T}_H| K\cap K'\neq \emptyset \mbox{~for some~}K'\in\mathcal{R}\}$. Also,  let $U_h$ and $U_H$, $\Sigma_h$ and $\Sigma_H$, $H_h(\Curl,\Omega)$ and $H_H(\Curl,\Omega)$ be finite element subspaces of $U$, $\Sigma$ and $H(\Curl,\Omega)$ defined on  $\mathcal{T}_h$ and $\mathcal{T}_H$,  respectively.

The mixed finite element method is to solve \eqref{eq:UnifiedProblem} in the  pair of the finite dimensional spaces $\Sigma_h\times U_h\subset\Sigma\times U$.
The corresponding discrete problem reads: Find $(\sigma_h, u_h)\in \Sigma_h\times U_h$ such that
\begin{equation}\label{discreteproblem}
\begin{split}
(\A\sigma_h, \tau_h)_{L^2(\Omega)}-(\ddiv\tau_h, u_h)_{L^2(\Omega)}=0,\text{ for any }\tau_h\in \Sigma_h,\\[0,5ex]
(\ddiv\sigma_h,v_h)_{L^2(\Omega)}-(f,v_h)_{L^2(\Omega)}=0, \text{ for any }v_h\in U_h.
\end{split}
\end{equation}

Let $\mathcal{Q}_h$ be the $L^2$-projection operator from $U$ onto the space $U_h$.  The edge or face error estimator with respect to a given subset $\mathcal{M}_h\subseteq\mathcal{T}_h$ is defined by,  \cite{carstensen2011priori, huang2012convergence},
\begin{equation*}
\eta^2(\sigma_h,\mathcal{M}_h):
=\left\{\begin{array}{ll}
\sum\limits_{K\in\mathcal{M}_h}\left(\|h_K\rot(\A\sigma_h)\|^2_{L^2(K)}+\sum\limits_{E\in K\cap \mathcal{E}_h}\|h_E^{1/2}[\A\sigma_h\cdot t_E]\|_{L^2(E)}^2\right) &d=2,\\
\sum\limits_{K\in\mathcal{M}_h}\left(\|h_K\Curl(\A\sigma_h)\|^2_{L^2(K)}+\sum\limits_{E\in K \cap \mathcal{E}_h}\|h_E^{1/2}[\A\sigma_h\times\nu_E]\|_{L^2(E)}^2\right) &d=3.
\end{array}
\right.
\end{equation*}
And given $f\in L^2(\Omega)$, define the data oscillation by
$$\osc^2(f,\mathcal{M}_h):=\sum\limits_{K\in\mathcal{M}_h}\|h_K(f-f_h)\|^2_{L^2(K)}\text{ with } f_h=\mathcal{Q}_h f.$$
It follows that
$$
\osc^2(f_h,\mathcal{T}_H)=\sum\limits_{K\in\mathcal{T}_H\backslash\mathcal{T}_h}\|h_K(f_h-f_H)\|^2_{L^2(K)}.
$$
For each $K\in \mathcal{T}_H$, since $\|h_K(f_h-f_H)\|_{L^2(K)}=\|h_K \mathcal{Q}_h(f-f_H)\|_{L^2(K)}\le\|h_K(f-f_H)\|_{L^2(K)}$,
$$\osc^2(f_h,\mathcal{T}_H)\le\sum\limits_{K\in\mathcal{T}_H\backslash\mathcal{T}_h}\|h_K(f-f_H)\|^2_{L^2(K)}=
\osc^2(f,\mathcal{T}_H\backslash\mathcal{T}_h).$$

Given $f\in L^2(\Omega)$, define an affine space $Z_h(f)$ by
\begin{equation}
Z_h(f):=\{\tau_h\in\Sigma_h, \ddiv\tau_h=f_h\}.
\end{equation}
In particular, $Z_h(0)$ is the kernel space of the discrete  divergence operator, which is also called the discrete divergence free space.
 Let ($\sigma_h,u_h)$ be the solution of
\eqref{discreteproblem}, we then have the following key property
\begin{equation}\label{otho1}
(\A\sigma_h, \tau_h)_{L^2(\Omega)}=0, \text{ for any }\tau_h\in Z_h(0).
\end{equation}

 We follow \cite[Lemma 2.1]{arnold1982interior} to  endow the space $U_h$ with the following discrete
 $H^1$ norm: for a given set $G$ consisting of elements $K\in\mathcal{T}_h$,
 \begin{equation}\label{eq:DiscreteNorm}
 \|v_h\|_{1,h,G}^2:=\sum\limits_{K\in G\cap\mathcal{T}_h}\|\nabla v_h\|_{L^2(K)}^2+\sum\limits_{E\in G\cap\mathcal{E}_h}h_E^{-1}\|[v_h]\|_{L^2(E)}^2.
 \end{equation}
 When $G=\Omega$, the subscript is omitted.  Hence the naturally matched  norm for the space $\Sigma_h$ is  the the $L^2$ norm.

Next, we propose five hypotheses, which are  sufficient for  convergence and optimality of AMFEMs.

 \begin{hypothesis}\label{H1}
 The discrete spaces $\Sigma_h$ and $U_h$ satisfy the following inclusion properties
$$
\Sigma_H\subset\Sigma_h\text{ and } \ddiv \Sigma_h \subset U_h.$$
\end{hypothesis}
\begin{hypothesis} \label{H2} The pair of spaces $(\Sigma_h,U_h)$  satisfies the discrete inf-sup condition
$$\|v_h\|_{1,h}\lesssim \sup\limits_{0\not=\tau_h\in\Sigma_h}\frac{(\ddiv\tau_h,v_h)_{L^2(\Omega)}}{\|\tau_h\|_{L^2(\Omega)}}, \text{ for any }
 v_h\in U_h,$$
 which  implies the following  equivalent inf-sup condition, see Brezzi and Fortin \cite{brezzi1991mixed} for more details,
\begin{equation}\label{einfsup}
\|\tau_h\|_{L^2(\Omega)/ Z_h(0)}\lesssim \sup\limits_{0\not=v_h\in U_h}\frac{(\ddiv\tau_h,v_h)_{L^2(\Omega)}}{\|v_h\|_{1,h}}, \text{ for any }
 \tau_h\in \Sigma_h.
 \end{equation}
 \end{hypothesis}

\begin{hypothesis}\label{H3}
Given $v_h\in U_h$, there exists an operator $\mathcal{S}_H: U_h\rightarrow U_H$ such that
\begin{equation}
\mathcal{S}_H v_h|_K=v_h|_K,\text{ for any } K\in \mathcal{T}_H\cap\mathcal{T}_h,
\end{equation}
and
\begin{equation}
\|v_h-\mathcal{S}_H v_h\|_{L^2(K)}\lesssim h_K \|v_h\|_{1,h,D_K}, \text{ for any } K\in \mathcal{T}_H\backslash\mathcal{T}_h,
\end{equation}
where $D_K:=\bigcup_{K'\in \mathcal{T}_h, K'\cap K\ne\emptyset}K'$.
\end{hypothesis}

\begin{hypothesis}\label{H4}
 Given $\xi_h\in\Sigma_h$ with $\ddiv\xi_h=0$, there exist $\varphi_h\in H_h(\Curl,\Omega)$ and
 an operator $\Pi_H: H_h(\Curl,\Omega)\rightarrow H_H(\Curl,\Omega)$ such that
  \begin{equation}\label{eq: CommutativeProperty}
 \xi_h=\Curl\varphi_h \text{~and~} \Curl \Pi_H \varphi_h\in\Sigma_H.
 \end{equation}
 Moreover, there exist $\psi\in H^1(\Omega)$ and $\phi\in L^2(\Omega)$ such that
 \begin{equation}\label{eq: decompositon}
\Curl(\varphi_h-\Pi_H\varphi_h)=\Curl\psi+\phi
\end{equation}
with $(\A\sigma_H,\phi)_{L^2(\Omega)}=0$ ($(\sigma_H, u_H)$ is
the solution to the discrete problem \eqref{discreteproblem} over $\mathcal{T}_H$) and
\begin{equation}\label{eq: Interpolation2}
\left\{\begin{array}{l}
\psi|_K=0,\quad\mbox{~for~any~} K\in\mathcal{T}_H\backslash\tilde{\mathcal{R}},\smallskip\\
\left(\sum_{K\in\mathcal{T}_H}\|h_K^{-1}\psi\|^2_{L^2(K)}\right)^{1/2}\lesssim  \|\Curl\varphi_h\|_{L^2(\Omega)},\smallskip\\
\left(\sum_{E\in\mathcal{E}_H}\|h_E^{-1/2}\psi\|^2_{L^2(E)}\right)^{1/2}\lesssim \|\Curl\varphi_h\|_{L^2(\Omega)}.
\end{array}\right.
\end{equation}

\end{hypothesis}

Besides these hypotheses on the finite element subspaces,
the a posteriori error estimator with reliability and efficiency is necessary in an adaptive algorithm, which will be described in the following Hypothesis 5.
\begin{hypothesis}\label{lemma:PostreioriEstimation}
Let $(\sigma,u)$ be the solution of \eqref{eq:UnifiedProblem} and $(\sigma_h,u_h)$ be the solution of \eqref{discreteproblem} over a triangulation $\mathcal{T}_h$,
 there exist constants $C_{Rel}$ and $C_{Eff}$ depending on the shape regularity of $\mathcal{T}_h$ such that
\begin{equation}\label{eq: UpperBound}
\|\sigma-\sigma_h\|^2_{\A}\le C_{Rel} (\eta^2(\sigma_h,\mathcal{T}_h)+\osc^2(f,\mathcal{T}_h)),  \quad\mbox{(Reliability)}
\end{equation}
\begin{equation}\label{eq: LowerBound}
C_{Eff}\eta^2(\sigma_h,\mathcal{T}_h)\le \|\sigma-\sigma_h\|^2_{\A}.  \quad\mbox{(Efficiency)}
\end{equation}
\end{hypothesis}

We will in Section 5 show that the Raviart--Thomas and the Brezzi--Douglas--Marini elements for the Poisson  and  Stokes problems satisfy
Hypotheses \ref{H1}--\ref{lemma:PostreioriEstimation} in both two and three dimensions.   We assume  in the next  two sections  that these five
hypotheses hold.

\section{Discrete reliability and quasi-orthogonality}
In this section we analyze discrete reliability of the estimator $\eta$, and also show quasi-orthogonality under the previous hypotheses.  Compared to
the analysis of both discrete reliability and  quasi-orthogonality in literature, see for instance, \cite{becker2008optimally,chen2009convergence, carstensen2011optimal, du2012error,huang2012convergence,carstensen2013quasi},  the novelty  of the analysis here
  is to  equip  the discrete displacement space $U_h$ with the discrete $H^1$ norm  defined in \eqref{eq:DiscreteNorm}.  Then it is natural to  endow the discrete stress space $\Sigma_h$ with the $L^2$ norm.  Moreover, it allows us to make use of the  equivalent form of the inf-sup condition \eqref{einfsup}.

\begin{thm}\label{thm:DiscreteUpperBound}
Given $f\in L^2(\Omega)$, let $(\sigma_h, u_h)$  and $(\sigma_H, u_H)$ be
the solutions to the discrete problem \eqref{discreteproblem} over the nested triangulations $\mathcal{T}_h$ and  $\mathcal{T}_H$ respectively. Then there exists a constant $C_{Drel}$ such that
\begin{equation}
\|\sigma_h-\sigma_H\|^2_{\A}\le C_{Drel}
\left(\eta^2(\sigma_H,\tilde{\mathcal{R}})+\osc^2(f, \mathcal{T}_H\backslash\mathcal{T}_h)\right).
\end{equation}
\end{thm}
\begin{proof}  The main idea is to evoke Hypothesis \ref{H2} and Hypothesis \ref{H4} to control  the discrete divergence free part and its orthogonal
  complementary, separately.  To this end, let $\xi_h$ be the $\A$ projection of $\sigma_H$ onto the space $Z_h(0)$  with
$$(\A\xi_h,\tau_h)_{L^2(\Omega)}=(\A\sigma_H,\tau_h)_{L^2(\Omega)}, \text{ for any } \tau_h\in Z_h(0).$$
By \eqref{otho1}, the error  $\|\sigma_h-\sigma_H\|^2_{\A}$ admits the following decomposition:
 \begin{equation}\label{0}
 \|\sigma_h-\sigma_H\|^2_{\A}=\|\sigma_h-\sigma_H+\xi_h-\xi_h\|^2_{\A}\lesssim \|\sigma_h-\sigma_H+\xi_h\|^2_{\A}+\|\xi_h\|^2_{\A}.
 \end{equation}
  Since
  $$(\A(\sigma_h-\sigma_H+\xi_h),\tau_h)_{L^2(\Omega)}=(\A\sigma_h,\tau_h)_{L^2(\Omega)}-(\A(\sigma_H-\xi_h),\tau_h)_{L^2(\Omega)}=0,$$
  for any  $\tau_h\in Z_h(0)$, it follows  from \eqref{coe} and Hypothesis \ref{H2} that
 \begin{equation}\label{3}
 \begin{split}
 \|\sigma_h-\sigma_H+\xi_h\|_{\A}&\lesssim\|\sigma_h-\sigma_H+\xi_h\|_{L^2(\Omega)/Z_h(0)}\\
 &\lesssim \sup\limits_{v_h\in U_h}\frac{\left(\ddiv(\sigma_h-\sigma_H+\xi_h),v_h\right)_{L^2(\Omega)}}{\|v_h\|_{1,h}}\\
 &=        \sup\limits_{v_h\in U_h}\frac{(f_h-f_H,v_h)_{L^2(\Omega)}}{\|v_h\|_{1,h}}.
 \end{split}
 \end{equation}
 In the last equation we use $\ddiv \sigma_h=f_h$, $\ddiv \sigma_H=f_H$, $\ddiv \xi_h=0$, which are direct results of Hypothesis \ref{H1}.
 Moreover,
\begin{equation}\label{4}
 \begin{split}
 (f_h-f_H,v_h)_{L^2(\Omega)}&=(f_h-f_H,v_h-\mathcal{S}_H v_h)_{L^2(\Omega)}\\
 &=\sum\limits_{K\in\mathcal{T}_H\backslash\mathcal{T}_h}(f_h-f_H,v_h-\mathcal{S}_H v_h)_{L^2(K)}\\
 \end{split}
 \end{equation}
 By Hypothesis \ref{H3}, a combination of \eqref{3} and \eqref{4} implies
 \begin{equation}\label{5}
 \|\sigma_h-(\sigma_H-\xi_h)\|^2_{\A}\lesssim\sum\limits_{K\in\mathcal{T}_H\backslash\mathcal{T}_h}\|h_K(f-f_H)\|^2_{L^2(K)}.
 \end{equation}
Next, we analyze the second term on the right-hand side of \eqref{0}. By the definition of $\xi_h$,  and \eqref{otho1},
\begin{equation}\label{eq: eq0}
 \|\xi_h\|_{\A}^2=(\A\xi_h,\xi_h)_{L^2(\Omega)}=(\A\sigma_H,\xi_h)_{L^2(\Omega)}=(\A\sigma_H,\xi_h-\tau_H)_{L^2(\Omega)},
 \end{equation}
 for any $\tau_H\in Z_H(0)$.  Since $\ddiv\xi_h=0$,  it follows from  Hypothesis \ref{H4} that there exists $\varphi_h\in H_h(\Curl,\Omega)$ such that $\xi_h=\Curl\varphi_h$. Hence the decomposition from \eqref{eq: decompositon} with  $\tau_H=\Curl \Pi_H\varphi_h\in\Sigma_H$ implies that there exist $\psi\in H(\Curl,\Omega)$ and $\phi\in L^2(\Omega)$ such that
 \begin{equation}\label{eq: eq1}
 \xi_h-\tau_H=\Curl\left(\varphi_h-\Pi_H\varphi_h\right)=\Curl\psi+\phi,
 \end{equation}
with $(\A\sigma_H,\phi)=0$ and $\psi$ satisfying \eqref{eq: Interpolation2}.
A summary of  \eqref{eq: eq0}, \eqref{eq: eq1} and \eqref{eq: Interpolation2} leads to
 \begin{equation}\label{divergencefree}
 \|\xi_h\|_{\A}^2=(\A\sigma_H,\Curl\psi)_{L^2(\Omega)}=(\A\sigma_H,\Curl\psi)_{\tilde{\mathcal{R}}}.
 \end{equation}
 We need to  estimate the term on the right hand side of \eqref{divergencefree}.  For the 2D case, an integration by parts plus \eqref{eq: Interpolation2} yield
\begin{equation*}
 \begin{split}
 \|\xi_h\|_{\A}^2
= &-\sum\limits_{K\in \tilde{\mathcal{R}}}\bigg(
       (\rot(\A\sigma_H),\psi)_{L^2(K)}+\sum_{E\in K\cap\mathcal{E}_H}\int_{E}\A\sigma_H\cdot t_E\psi ds\bigg)\\
\lesssim &\left(\sum\limits_{K\in\tilde{\mathcal{R}}}\bigg(\|h_K\rot(\A\sigma_H)\|^2_{L^2(K)}
      +\sum\limits_{E\in K\cap\mathcal{E}_H}\|h_E^{1/2}[\A\sigma_H\cdot t_E]\|_{L^2(E)}^2\bigg)\right)^{1/2}\|\xi_h\|_{L^2(\Omega)}.
 \end{split}
 \end{equation*}
 Since $\ddiv\xi_h=0$, it follows from \eqref{coe} that
 $$
 \|\xi_h\|_{L^2(\Omega)}\lesssim\|\xi_h\|_{\A},
 $$
 which  immediately implies
 \begin{equation}\label{6-1}
 \|\xi_h\|^2_{\A}\lesssim \sum\limits_{K\in\tilde{\mathcal{R}}}\bigg(\|h_K\rot(\A\sigma_H)\|^2_{L^2(K)}
      +\sum\limits_{E\in K\cap\mathcal{E}_H}\|h_E^{1/2}[\A\sigma_H\cdot t_E]\|_{L^2(E)}^2\bigg).
 \end{equation}
For the 3D case, a similar argument shows
\begin{equation*}
 \begin{split}
 \|\xi_h\|_{\A}^2
= &-\sum\limits_{K\in\tilde{\mathcal{R}}}\bigg(
       (\Curl(\A\sigma_H),\psi)_{L^2(K)}+\sum_{E\in K\cap\mathcal{E}_H}\int_{E}\A\sigma_H\times \nu_E\psi ds\bigg)\\
\lesssim &\left(\sum\limits_{K\in\tilde{\mathcal{R}}}\bigg(\|h_K\Curl(\A\sigma_H)\|^2_{L^2(K)}
      +\sum\limits_{E\in K\cap\mathcal{E}_H}\|h_E^{1/2}[\A\sigma_H\times\nu_E]\|_{L^2(E)}^2\bigg)\right)^{1/2}\|\xi_h\|_{\A}.
 \end{split}
 \end{equation*}
 Therefore,
 \begin{equation}\label{6-2}
  \|\xi_h\|^2_{\A}\lesssim \sum\limits_{K\in\tilde{\mathcal{R}}}\bigg(\|h_K\Curl(\A\sigma_H)\|^2_{L^2(K)}
      +\sum\limits_{E\in K\cap\mathcal{E}_H}\|h_E^{1/2}[\A\sigma_H\times\nu_E]\|_{L^2(E)}^2\bigg).
 \end{equation}
 Finally,  the desired result follows from \eqref{0}, \eqref{5}, \eqref{6-1} and \eqref{6-2}.
\end{proof}

To establish quasi-orthogonality, we follow the idea of  \cite{chen2009convergence} to introduce the following problem: Find $(\tilde{\sigma}_h,\tilde{u}_h)\in\Sigma_h\times U_h$ such that
\begin{equation}\label{discreteproblem-inter}
\begin{split}
(\A\tilde{\sigma}_h, \tau_h)_{L^2(\Omega)}-(\ddiv\tau_h, \tilde{u}_h)_{L^2(\Omega)}=0,\text{ for any }\tau_h\in \Sigma_h,\\[0,5ex]
(\ddiv\tilde{\sigma}_h,v_h)_{L^2(\Omega)}-(f_H,v_h)_{L^2(\Omega)}=0, \text{ for any }v_h\in U_h.
\end{split}
\end{equation}

\begin{lem}\label{lemma:DiscreteStablity}
Given $f\in L^2(\Omega)$, let $(\sigma, u)$ be the solution of \eqref{eq:UnifiedProblem}, $(\sigma_h, u_h)$  and $(\sigma_H, u_H)$ be
the solutions to the discrete problem \eqref{discreteproblem} over the nested triangulations $\mathcal{T}_h$ and  $\mathcal{T}_H$ respectively,
and let $(\tilde{\sigma}_h,\tilde{u}_h)$ be the solution of \eqref{discreteproblem-inter}.
Then
\begin{equation}\label{eq:DiscreteStablity}
(\A(\sigma-\sigma_h),\tilde{\sigma}_h-\sigma_H)_{L^2(\Omega)}=0,
\end{equation}
\begin{equation}\label{Quasi-Orthogonality0}
\|\sigma_h-\tilde{\sigma}_h\|_{\A}\le \sqrt{C_0} \osc(f, \mathcal{T}_H\backslash\mathcal{T}_h).
\end{equation}
\end{lem}
\begin{proof}
By  Hypothesis \ref{H1}, the definitions of $\sigma$, $\sigma_h$, $\tilde{\sigma}_h$ and $\sigma_H$ imply \eqref{eq:DiscreteStablity} directly via
\begin{equation*}
\begin{split}
(\A(\sigma-\sigma_h),\tilde{\sigma}_h-\sigma_H)_{L^2(\Omega)}
&=(u-u_h,\ddiv(\tilde{\sigma}_h-\sigma_H))_{L^2(\Omega)}\\
&=(u-u_h,f_H-f_H)_{L^2(\Omega)}=0,
\end{split}
\end{equation*}
which shows \eqref{eq:DiscreteStablity}.

Since by \eqref{otho1} $(\A\sigma_h,\tau_h)_{L^2(\Omega)}=0$ and $(\A\tilde{\sigma}_h,\tau_h)_{L^2(\Omega)}=0$ for any $\tau_h \in Z_h(0)$,
\begin{equation}
(\A(\sigma_h-\tilde{\sigma}_h),\tau_h)_{L^2(\Omega)}=0, \text{ for any } \tau_h\in Z_h(0),
\end{equation}
which, along with the relations $\ddiv\sigma_h=f_h$, $\ddiv\tilde{\sigma}_h=f_H$ which follow from  Hypothesis \ref{H1}, the estimate \eqref{4} through Hypothesis \ref{H3}, implies
\begin{equation}
 \begin{split}
 \|\sigma_h-\tilde{\sigma}_h\|_{\A}&\lesssim\|\sigma_h-\tilde{\sigma}_h\|_{L^2(\Omega)/Z_h(0)}\\
 &\lesssim \sup\limits_{v_h\in U_h}\frac{(\ddiv(\sigma_h-\tilde{\sigma}_h),v_h)_{L^2(\Omega)}}{\|v_h\|_{1,h}}\\
 &=        \sup\limits_{v_h\in U_h}\frac{(f_h-f_H,v_h)_{L^2(\Omega)}}{\|v_h\|_{1,h}}\\
 &\lesssim \bigg(\sum\limits_{K\in\mathcal{T}_H\backslash\mathcal{T}_h}\|h_K(f-f_H)\|^2_{L^2(K)}\bigg)^{1/2}.
 \end{split}
 \end{equation}
 This proves \eqref{Quasi-Orthogonality0}.
\end{proof}
%

\begin{thm}\label{thm:Quasi-Orthogonality}
Given $f\in L^2(\Omega)$, let $(\sigma, u)$ be the solution of \eqref{eq:UnifiedProblem}, $(\sigma_h, u_h)$  and $(\sigma_H, u_H)$ be
the solutions to the discrete problem \eqref{discreteproblem} over the nested triangulations $\mathcal{T}_h$ and  $\mathcal{T}_H$ respectively.
Then
\begin{equation}\label{Quasi-Orthogonality1}
(\A(\sigma-\sigma_h),\sigma_h-\sigma_H)_{L^2(\Omega)}\le \sqrt{C_0}\|\sigma-\sigma_h\|_{\A}\osc(f, \mathcal{T}_H\backslash\mathcal{T}_h).
\end{equation}
Thus, for any $\delta\ge 0$,
\begin{equation}\label{Quasi-Orthogonality2}
 (1-\delta)\|\sigma-\sigma_h\|^2_{\A}\le \|\sigma-\sigma_H\|^2_{\A}-\|\sigma_h-\sigma_H\|^2_{\A}+\frac{C_0}{\delta}\osc^2(f, \mathcal{T}_H\backslash\mathcal{T}_h).
\end{equation}
\end{thm}
\begin{proof} We follow the idea of  \cite[Theorem 3.2]{chen2009convergence}.
Let $(\tilde{\sigma}_h, \tilde{u}_h)$ be the solution of \eqref{discreteproblem-inter}. By Hypothesis \ref{H1} and  Lemma \ref{lemma:DiscreteStablity},
$$
(\A(\sigma-\sigma_h),\tilde{\sigma}_h-\sigma_H)_{L^2(\Omega)}=0.
$$
Thus,
\begin{equation}
 (\A(\sigma-\sigma_h),\sigma_h-\sigma_H)_{L^2(\Omega)}=(\A(\sigma-\sigma_h),\sigma_h-\tilde{\sigma}_h)_{L^2(\Omega)}
 \le \|\sigma-\sigma_h\|_{\A} \|\sigma_h-\tilde{\sigma}_h\|_{\A}.
\end{equation}
Therefore, the estimate \eqref{Quasi-Orthogonality1} follows from the inequality (\ref{Quasi-Orthogonality0}).
By the identity $\sigma-\sigma_H=\sigma-\sigma_h+\sigma_h-\sigma_H$,
$$\|\sigma-\sigma_H\|^2_{\A}=\|\sigma-\sigma_h\|^2_{\A}+\|\sigma_h-\sigma_H\|^2_{\A}+2(\A(\sigma-\sigma_h),\sigma_h-\sigma_H).$$
In general, we use
\begin{equation}
 \begin{split}
 \|\sigma-\sigma_H\|^2_{\A}&=\|\sigma-\sigma_h\|^2_{\A}+\|\sigma_h-\sigma_H\|^2_{\A}+2(\A(\sigma-\sigma_h),\sigma_h-\sigma_H)\\
 &\geq \|\sigma-\sigma_h\|^2_{\A}+\|\sigma_h-\sigma_H\|^2_{\A}-2\sqrt{C_0}\|\sigma-\sigma_h\|_{\A}\osc(f,\mathcal{T}_H\backslash\mathcal{T}_h)\\
 &\geq \|\sigma_h-\sigma_H\|^2_{\A}+(1-\delta)\|\sigma-\sigma_h\|^2_{\A}-\frac{C_0}{\delta}\osc^2(f,\mathcal{T}_H\backslash\mathcal{T}_h)
 \end{split}
 \end{equation}
 to prove (\ref{Quasi-Orthogonality2}). In the last step, we have used the Young inequality.
\end{proof}

\section{Convergence and Optimality of AMFEM}
In this section, we  prove  convergence and optimality of the adaptive mixed finite element methods.
First we present the adaptive algorithms.
In what follows, we replace the dependence on the actual mesh $\mathcal{T}$ by the iteration counter $k$.

{\bf Algorithm.} Given an initial mesh $\mathcal{T}_0$ and a marking parameter $0<\theta<1$, set $k=0$ and iterate
\begin{itemize}
\item{Solve on $\mathcal{T}_k$, to get the solution $\sigma_k$.}
\item{Compute the error estimator $\eta=\eta(\sigma_k,\mathcal{T}_k)$.}
\item{Mark the minimal element set $\mathcal{M}_k$, such that
              $$\eta^2(\sigma_k,\mathcal{M}_k)+\osc^2(f,\mathcal{M}_k)\ge\theta(\eta^2(\sigma_k,\mathcal{T}_k)+\osc^2(f,\mathcal{T}_k)).$$}
\item{Refine each element $K\in \mathcal{M}_k$ by the newest vertex bisection, and make some necessary  completeness  to get a refined  conforming mesh $\mathcal{T}_{k+1}$; $k=k+1$.}
\end{itemize}

\subsection{Convergence}
 We follow  the arguments in \cite{cascon2008quasi, chen2009convergence, huang2012convergence} to prove convergence of the above adaptive algorithms.

\begin{lem} Given $f\in L^2(\Omega)$, let $(\sigma, u)$ be the solution of \eqref{eq:UnifiedProblem}, $(\sigma_k, u_k)$  and $(\sigma_{k-1}, u_{k-1})$ be
the solutions to the discrete problem \eqref{discreteproblem} over the nested triangulations $\mathcal{T}_k$ and  $\mathcal{T}_{k-1}$ respectively. Then given any positive constant $\epsilon$, there exist positive constants $0<\rho<1$ (depending on the dimension)
and $\beta_2(\epsilon)$ such that
\begin{equation}\label{thm:Twolevel}
\eta^2(\sigma_k,\mathcal{T}_k)\le (1+\epsilon)(\eta^2(\sigma_{k-1},\mathcal{T}_{k-1})-\rho\eta^2(\sigma_{k-1},\mathcal{M}_{k-1}))+\frac1{\beta_2(\epsilon)}\|\sigma_k-\sigma_{k-1}\|^2_{\A},
\end{equation}
and
\begin{equation}\label{lemma: Oscilation}
\osc^2(f,\mathcal{T}_k)\le\osc^2(f,\mathcal{T}_{k-1})-\rho\osc^2(f,\mathcal{T}_{k-1}\backslash\mathcal{T}_k).
\end{equation}
\end{lem}
\begin{proof}
By the definition of $\eta^2(\sigma_k,\mathcal{T}_{k})$, $\eta^2(\sigma_{k-1},\mathcal{T}_{k})$, the trace theorem and the inverse inequality imply
\begin{equation*}
|\eta(\sigma_k,\mathcal{T}_{k})-\eta(\sigma_{k-1},\mathcal{T}_{k})|\lesssim\|\sigma_k-\sigma_{k-1}\|_{\A}.
\end{equation*}
An application of the Young inequality yields
\begin{equation}\label{1}
\eta^2(\sigma_k,\mathcal{T}_{k})\le(1+\epsilon)\eta^2(\sigma_{k-1},\mathcal{T}_{k})+\frac1{\beta_2(\epsilon)}\|\sigma_k-\sigma_{k-1}\|^2_{\A}.
\end{equation}

Let $\mathcal{N}_{k}=\mathcal{T}_{k}\backslash\mathcal{T}_{k-1}$ be the set of the new elements in $\mathcal{T}_k$ but not in $\mathcal{T}_{k-1}$,   and  $\bar{\mathcal{M}}_{k-1}\subseteq\mathcal{T}_{k-1}$
be the set of the  elements which are refined.  Notice that $\mathcal{T}_{k-1}\backslash\bar{\mathcal{M}}_{k-1}=\mathcal{T}_{k}\backslash\mathcal{N}_{k}$.
Given element $K\in\mathcal{N}_k$, consider its edge/face $E\in K\cap\mathcal{E}_k$.
If $E$ is in the interior of some element $T\in\mathcal{M}_{k-1}$, then $[\A\sigma_{k-1}\times\nu_E]|_E=0$ since $\sigma_{k-1}$ is a polynomial in $K$;
otherwise, its measure is at most  half of that of some edge/face of $T\in\mathcal{M}_{k-1}$ and thus
$$\eta^2(\sigma_{k-1},\mathcal{N}_k)\le 2^{-\frac1{d-1}}\eta^2(\sigma_{k-1},\bar{\mathcal{M}}_{k-1}).$$
Since some more elements are refined for the conformity of the triangulation, $\mathcal{M}_{k-1}\subseteq\bar{\mathcal{M}}_{k-1}$. Therefore,
\begin{eqnarray*}
\eta^2(\sigma_{k-1},\mathcal{T}_k)
&=&\eta^2(\sigma_{k-1},\mathcal{N}_k)+\eta^2(\sigma_{k-1},\mathcal{T}_k\backslash\mathcal{N}_k)\\
&\le&2^{-\frac1{d-1}}\eta^2(\sigma_{k-1},\bar{\mathcal{M}}_{k-1})+\eta^2(\sigma_{k-1},\mathcal{T}_{k-1}\backslash\bar{\mathcal{M}}_{k-1})\\
&\le&\eta^2(\sigma_{k-1},\mathcal{T}_{k-1})-(1-2^{-\frac1{d-1}})\eta^2(\sigma_{k-1},\bar{\mathcal{M}}_{k-1})\\
&\le&\eta^2(\sigma_{k-1},\mathcal{T}_{k-1})-(1-2^{-\frac1{d-1}})\eta^2(\sigma_{k-1},\mathcal{M}_{k-1}).
\end{eqnarray*}
This  leads to
\begin{equation}\label{2}
\eta^2(\sigma_{k-1},\mathcal{T}_k)\le \eta^2(\sigma_{k-1},\mathcal{T}_{k-1})-(1-2^{-\frac1{d-1}})\eta^2(\sigma_{k-1},\mathcal{M}_{k-1}).
\end{equation}
With $\rho:=1-2^{-\frac1{d-1}}$, a combination of \eqref{1} and \eqref{2} proves  \eqref{thm:Twolevel}.
As for  \eqref{lemma: Oscilation}, it is an immediate  result of the definition of the mesh size $h_K$.
\end{proof}

In the next  theorem, we  establish  convergence of the adaptive methods.
\begin{thm}\label{thm:Convergence}
Given $f\in L^2(\Omega)$, let $(\sigma, u)$ be the solution of \eqref{eq:UnifiedProblem}, $(\sigma_k, u_k)$  and $(\sigma_{k-1}, u_{k-1})$ be
the solutions to the discrete problem \eqref{discreteproblem} over the nested triangulations $\mathcal{T}_k$ and  $\mathcal{T}_{k-1}$ respectively.
Then there exist positive constants $0<\alpha<1$,  $\beta>0$, $\gamma>0$ such that
$$\epsilon_k\le\alpha\epsilon_{k-1},$$
where
$$\epsilon_k=\|\sigma-\sigma_k\|_{\A}^2+\gamma\eta^2(\sigma_k,\mathcal{T}_k)+(\beta+\gamma)\osc^2(f,\mathcal{T}_{k}).$$
\end{thm}
\begin{proof}
It follows from  Theorem \ref{thm:Quasi-Orthogonality}, \eqref{thm:Twolevel} and \eqref{lemma: Oscilation} that
\begin{eqnarray*}
\epsilon_k
&\le&\frac{1}{1-\delta}\|\sigma-\sigma_{k-1}\|^2_{\A}-\frac{1}{1-\delta}\|\sigma_k-\sigma_{k-1}\|^2_{\A}+\frac{C_0}{\delta(1-\delta)}\osc^2(f,\mathcal{T}_{k-1}\backslash\mathcal{T}_k) \\
&&+\gamma(1+\epsilon)(\eta^2(\sigma_{k-1},\mathcal{T}_{k-1})-\rho\eta^2(\sigma_{k-1},\mathcal{M}_{k-1}))+\frac{\gamma}{\beta_2(\epsilon)}\|\sigma_k-\sigma_{k-1}\|^2_{\A}\\
&&+(\beta+\gamma)(\osc^2(f,\mathcal{T}_{k-1})-\rho\osc^2(f,\mathcal{T}_{k-1}\backslash\mathcal{T}_k)).
\end{eqnarray*}
Now the choice of $\gamma=\frac{\beta_2(\epsilon)}{1-\delta}$ and $\beta=\frac{C_0}{\rho\delta(1-\delta)}$  leads to
\begin{eqnarray*}
\epsilon_k
&\le&\frac{1}{1-\delta}\|\sigma-\sigma_{k-1}\|^2_{\A}+\beta\osc^2(f,\mathcal{T}_{k-1})\\
&&+\gamma(1+\epsilon)\big(\eta^2(\sigma_{k-1},\mathcal{T}_{k-1})+\osc^2(f,\mathcal{T}_{k-1})\big)\\
&&-\gamma(1+\epsilon)\rho\big(\eta^2(\sigma_{k-1},\mathcal{M}_{k-1})+\osc^2(f,\mathcal{T}_{k-1}\backslash\mathcal{T}_k)\big).
\end{eqnarray*}
Since $\mathcal{M}_{k-1}\subset\mathcal{T}_{k-1}\backslash\mathcal{T}_k$, the marking strategy in adaptive Algorithm implies that
$$\eta^2(\sigma_{k-1},\mathcal{M}_{k-1})+\osc^2(f,\mathcal{T}_{k-1}\backslash\mathcal{T}_k)
\ge\theta(\eta^2(\sigma_k,\mathcal{T}_{k-1})+\osc^2(f,\mathcal{T}_{k-1})).$$
A substitution of this inequality into the  previous one yields
\begin{eqnarray*}
\epsilon_k
&\le&\frac{1}{1-\delta}\|\sigma-\sigma_{k-1}\|^2_{\A}+\beta\osc^2(f,\mathcal{T}_{k-1})\\
&& +\gamma(1+\epsilon)(1-\rho\theta)(\eta^2(\sigma_{k-1},\mathcal{T}_{k-1})+\osc^2(f,\mathcal{T}_{k-1})).
\end{eqnarray*}
By the definition of $\epsilon_{k-1}$, we have for any $0<\alpha<1$,
\begin{eqnarray*}
\epsilon_k-\alpha \epsilon_{k-1}
&\le&(\frac{1}{1-\delta}-\alpha)\|\sigma-\sigma_{k-1}\|^2_{\A}+\beta(1-\alpha)\osc^2(f,\mathcal{T}_{k-1})\\
&&+\gamma((1+\epsilon)(1-\rho\theta)-\alpha)(\eta^2(\sigma_{k-1},\mathcal{T}_{k-1})+\osc^2(f,\mathcal{T}_{k-1})).
\end{eqnarray*}
 Hypothesis \ref{lemma:PostreioriEstimation} states
  $$
  \|\sigma-\sigma_{k-1}\|^2_{\A}\le C_{Rel}(\eta^2(\sigma_{k-1},\mathcal{T}_{k-1})+\osc^2(f,\mathcal{T}_{k-1})).$$
This inequality plus the inequality give
\begin{eqnarray*}
\epsilon_k-\alpha \epsilon_{k-1}
&\le&\bigg((\frac{1}{1-\delta}-\alpha)C_{Rel}+\gamma(1+\epsilon)(1-\rho\theta)-\gamma\alpha\bigg)\eta^2(\sigma_{k-1},\mathcal{T}_{k-1})\\
&&+\bigg((\frac{1}{1-\delta}-\alpha)C_{Rel}+\gamma(1+\epsilon)(1-\rho\theta)-\gamma\alpha+\beta(1-\alpha)\bigg)\osc^2(f,\mathcal{T}_{k-1}).
\end{eqnarray*}
To ensure $\epsilon_k-\alpha \epsilon_{k-1}\le 0$,  the factor $\alpha$ can be chosen  such that
$$
(\frac{1}{1-\delta}-\alpha)C_{Rel}+\gamma(1+\epsilon)(1-\rho\theta)-\gamma\alpha+\beta(1-\alpha)\le 0,
$$
which implies  $\alpha=\frac{\beta+\gamma(1+\epsilon)(1-\rho\theta)+\frac{C_{Rel}}{1-\delta}}{\beta+\gamma+C_{Rel}}$ with $0<\delta<\frac{\gamma(\rho\theta-\epsilon(1-\rho\theta))}{\gamma(\rho\theta-\epsilon(1-\rho\theta))+C_{Rel}}$.
\end{proof}

\subsection{Optimality}
Let $\mathcal{T}_0$ be an initial quasi-uniform triangulation with $\# \mathcal{T}_0>2$, and let $\mathbb{T}_N$ be the set of all
possible triangulations $\mathcal{T}$ which is generated from $\mathcal{T}_0$ with at most $N$ elements more than $\mathcal{T}_0$. For $s>0$ we define
the approximation class $\mathbb{A}_s$ as
$$\mathbb{A}_s:=\{(\sigma,f):|\sigma,f|_s<\infty, \mbox{~with~} |\sigma,f|_s:=\sup_{N>0}(N^s\inf_{\mathcal{T}\in\mathbb{T}_N}\inf_{\tau\in\Sigma_\mathcal{T}}\|\sigma-\tau\|_{\A}^2+\osc^2(f,\mathcal{T}))\}.$$

\begin{lem}\label{lemma: EnergyError}
Given a parameter
\begin{equation}\label{eq:MarkingCondition}
\theta\in \left(0,\frac{\min\{C_{Eff},1\}}{C_{Drel}+\min\{C_{Eff},1\}+1}\right),
\end{equation}
let $(\sigma, u)$ be the solution of \eqref{eq:UnifiedProblem}, $(\sigma_h, u_h)$ and $(\sigma_H, u_H)$ be
the solutions to the discrete problem \eqref{discreteproblem} over $\mathcal{T}_h$ and $\mathcal{T}_H$, satisfying
\begin{equation}\label{condition}
\|\sigma-\sigma_h\|_{\A}^2+\osc^2(f,\mathcal{T}_h)\le \alpha'(\|\sigma-\sigma_H\|_{\A}^2+\osc^2(f,\mathcal{T}_H))
\end{equation}
with $0<\alpha'<\frac{\min\{C_{Eff},1\}-(\min\{C_{Eff},1\}+C_{Drel}+1)\theta}{\min\{C_{Eff},1\}+C_0}\in (0,1)$, then it holds
$$\theta(\eta^2(\sigma_H,\mathcal{T}_H)+\osc^2(f,\mathcal{T}_H))\le\eta^2(\sigma_H,\tilde{\mathcal{R}})+\osc^2(f,\mathcal{T}_H\backslash\mathcal{T}_h).$$
\end{lem}
\begin{proof}

On one hand, from Theorem \ref{thm:DiscreteUpperBound} it holds
\begin{equation}\label{eq:le}
\|\sigma_h-\sigma_H\|_{\A}^2\le C_{Drel}(\eta^2(\sigma_H,\tilde{\mathcal{R}})+\osc^2(f,\mathcal{T}_H\backslash\mathcal{T}_h)).
\end{equation}
On the other hand, from Theorem \ref{thm:Quasi-Orthogonality} and the Young inequality it holds
\begin{eqnarray*}
2(\sigma-\sigma_h,\sigma_h-\sigma_H)_{\mathcal{A}}
&\le&2\sqrt{C_0}\|\sigma-\sigma_h\|_{\A}\osc(f,\mathcal{T}_H\backslash\mathcal{T}_h)\\
&\le& \frac{C_0}{\min\{C_{Eff},1\}}\|\sigma-\sigma_h\|^2_{\A}+\min\{C_{Eff},1\}\osc^2(f,\mathcal{T}_H\backslash\mathcal{T}_h).
\end{eqnarray*}
This leads to
\begin{eqnarray*}
\|\sigma_h-\sigma_H\|_{\A}^2&=&\|\sigma-\sigma_H\|^2_{\A}-\|\sigma-\sigma_h\|^2_{\A}-2(\sigma-\sigma_h,\sigma_h-\sigma_H)_{\A}\\
&\ge&(\|\sigma-\sigma_H\|_{\A}^2+\osc^2(f,\mathcal{T}_H))\\
&&-(1+\frac{C_0}{\min\{C_{Eff},1\}})(\|\sigma-\sigma_h\|_{\A}^2+\osc^2(f,\mathcal{T}_h))\\
&&-\osc^2(f,\mathcal{T}_H)+\osc^2(f,\mathcal{T}_h)-\min\{C_{Eff},1\}\osc^2(f,\mathcal{T}_H\backslash\mathcal{T}_h).
\end{eqnarray*}
The condition \eqref{condition}, the lower bound in Hypothesis \ref{lemma:PostreioriEstimation},  and the relation
$$
|\osc^2(f,\mathcal{T}_H)-\osc^2(f,\mathcal{T}_h)|\le\osc^2(f,\mathcal{T}_H\backslash\mathcal{T}_h)
$$
imply
\begin{eqnarray}
\|\sigma_h-\sigma_H\|_{\A}^2
&\ge &\left(\min\{C_{Eff},1\}-(\min\{C_{Eff},1\}+C_0)\alpha'\right)\nonumber\\
&&\times(\eta^2(\sigma_H,\mathcal{T}_H)+\osc^2(f,\mathcal{T}_H))\label{eq:ge}\\
&&-(\min\{C_{Eff},1\}+1)\osc^2(f,\mathcal{T}_H\backslash\mathcal{T}_h).\nonumber
\end{eqnarray}
A combination of \eqref{eq:le} and \eqref{eq:ge} yields
\begin{eqnarray*}
&&(C_{Drel}+\min\{C_{Eff},1\}+1)(\eta^2(\sigma_H,\tilde{\mathcal{R}})+\osc^2(f,\mathcal{T}_H\backslash\mathcal{T}_h))\\
& &\ge (\min\{C_{Eff},1\}-(\min\{C_{Eff},1\}+C_0)\alpha')(\eta^2(\sigma_H,\mathcal{T}_H)+\osc^2(f,\mathcal{T}_H)),
\end{eqnarray*}
from which, and the definition of $\alpha'$ and the restriction on $\theta$, we  obtain the desired result.
\end{proof}

\begin{lem} \label{lemma:Stability}
Let $(\sigma, u)$ be the solution of \eqref{eq:UnifiedProblem}, $(\sigma_h, u_h)$ and $(\sigma_H, u_H)$ be
the solutions to the discrete problem \eqref{discreteproblem} over $\mathcal{T}_h$ and $\mathcal{T}_H$,
 there exists a constant
$C_1>0$ such that
\begin{equation}\label{eq:stability}
\|\sigma-\sigma_h\|_{\A}^2+\osc^2(f,\mathcal{T}_h)\le C_1(\|\sigma-\sigma_H\|_{\A}^2+\osc^2(f,\mathcal{T}_H)).
\end{equation}
\end{lem}
\begin{proof}
From \eqref{Quasi-Orthogonality2} and \eqref{lemma: Oscilation}, for any $0<\delta<1$, it holds
\begin{eqnarray*}
 &&\|\sigma-\sigma_h\|^2_{\A}+\osc^2(f,\mathcal{T}_h)\\
 &&\le \frac1{1-\delta}\|\sigma-\sigma_H\|^2_{\A} +\frac{C_0}{\delta(1-\delta)}\osc^2(f, \mathcal{T}_H\backslash\mathcal{T}_h)+\osc^2(f, \mathcal{T}_h)\\
 &&\le \frac1{1-\delta}\|\sigma-\sigma_H\|^2_{\A} +\left(\frac{C_0}{\delta(1-\delta)}+1\right)\osc^2(f, \mathcal{T}_H),
\end{eqnarray*}
which implies the desired result.
\end{proof}

\begin{thm}\label{estimateMk}
Let $\mathcal{M}_k$ be a set of marked elements with minimal cardinality, $(\sigma, u)$ the solution of \eqref{eq:UnifiedProblem}, and $(\mathcal{T}_k,\Sigma_k,\sigma_k, u_k)$
the sequence of triangulations, finite element spaces and discrete solutions produced by the adaptive finite element methods with the marking parameter $\theta$ in Lemma \ref{lemma: EnergyError}. It holds that
$$\#\mathcal{M}_k\le (\alpha')^{-\frac1s}|u,f|_{s}^{\frac1s}C_1^\frac1sC_2(\|\sigma-\sigma_k\|_{\A}^2+\osc^2(f,\mathcal{T}_k))^{-\frac1s},$$
where $\alpha'$ is defined in Lemma \ref{lemma: EnergyError}, $C_1$  in Lemma \ref{lemma:Stability},  and $C_2$ only depends on the shape regularity of $\mathcal{T}_0$.
\end{thm}
\begin{proof}
We set $\varepsilon=\alpha'C_1^{-1}(\|\sigma-\sigma_k\|_{\A}^2+\osc^2(f,\mathcal{T}_k))$ where $\alpha'$ is from Lemma \ref{lemma: EnergyError}, and $C_1$ from Lemma \ref{lemma:Stability}. Since $(\sigma,f)\in \mathbb{A}_s$, there exist a refinement of $\mathcal{T}_0$, say,  $\mathcal{T}_\varepsilon$,  and $\sigma_\varepsilon\in \Sigma_{\mathcal{T}_\varepsilon}$ such that
$$\#\mathcal{T}_\varepsilon-\#\mathcal{T}_0\le|\sigma,f|_s^{\frac1s}\varepsilon^{-\frac1s},$$
$$\|\sigma-\sigma_\varepsilon\|_{\A}^2+\osc^2(f,\mathcal{T}_\varepsilon)\le\varepsilon.$$
Let $\mathcal{T}_*$ be the overlay of $\mathcal{T}_\varepsilon$ and $\mathcal{T}_k$, and  $(\sigma_\ast, u_\ast)$ be the corresponding discrete solution on $\mathcal{T}_*$.
Since $\mathcal{T}_*$ is a refinement of $\mathcal{T}_\varepsilon$, it follows from Lemma \ref{lemma:Stability} that
$$\|\sigma-\sigma_*\|_{\A}^2+\osc^2(f,\mathcal{T}_*)\le C_1(\|\sigma-\sigma_\varepsilon\|_{\A}^2+\osc^2(f,\mathcal{T}_\varepsilon))\le C_1\varepsilon=\alpha'(\|\sigma-\sigma_k\|_{\A}^2+\osc^2(f,\mathcal{T}_k)).$$
From Lemma \ref{lemma: EnergyError} it holds
$$\theta(\eta^2(\sigma_k,\mathcal{T}_k)+\osc^2(f,\mathcal{T}_k))\le\eta^2(\sigma_k,\widetilde{\mathcal{T}_k\backslash\mathcal{T}_*})+\osc^2(f,\mathcal{T}_k\backslash\mathcal{T}_*),$$
here $\widetilde{\mathcal{T}_k\backslash\mathcal{T}_*}$ is similarly defined as  $\tilde{\mathcal{R}}$. Note that the marking step in the adaptive Algorithm with $\theta$ chooses a subset of $\mathcal{M}_k\subset\mathcal{T}_k$ with minimal cardinality so that  the same property holds.
Therefore, there exists a constant $C_2$ depending on the shape regularity of $\mathcal{T}_0$ such that
$$\#\mathcal{M}_k\le \#\widetilde{\mathcal{T}_k\backslash\mathcal{T}_*}\le C_2(\#\mathcal{T}_*-\#\mathcal{T}_k)\le C_2(\#\mathcal{T}_\varepsilon-\#\mathcal{T}_0).$$
By the definition of $\varepsilon$, a combination of the above inequalities  shows
$$\#\mathcal{M}_k\le (\alpha')^{-\frac1s}|\sigma,f|_{s}^{\frac1s}C_1^\frac1sC_2(\|\sigma-\sigma_k\|_{\A}^2+\osc^2(f,\mathcal{T}_k))^{-\frac1s}.$$
\end{proof}

\begin{thm}
Let $\mathcal{M}_k$ be a set of marked elements with minimal cardinality, $(\sigma,u)$ the solution of \eqref{eq:UnifiedProblem}, and $(\mathcal{T}_k,\Sigma_k,\sigma_k,u_k)$
the sequence of triangulations, finite element spaces and discrete solutions produced by the adaptive finite element methods with the marking parameter $\theta$ in Lemma \ref{lemma: EnergyError}. Then it holds that
$$\|\sigma-\sigma_N\|_{\A}^2+\osc^2(f,\mathcal{T}_N)\lesssim |\sigma, f|_{s}(\#\mathcal{T}_N-\#\mathcal{T}_0)^{-s},\quad \mbox{for~~} (\sigma,f)\in\mathbb{A}_s.$$

\end{thm}
\begin{proof}
Let $\mu=(\alpha')^{-\frac1s}|u,f|_s^{\frac1s}C_1^\frac1sC_2$. We use the result  $\#\mathcal{T}_k-\#\mathcal{T}_0\lesssim\sum_{j=0}^{k-1}\#\mathcal{M}_j$
from \cite{stevenson2007optimality}, and the upper bound of $\# \mathcal{M}_j$ in Theorem \ref{estimateMk},  to obtain that
$$\#\mathcal{T}_N-\#\mathcal{T}_0\lesssim\sum_{j=0}^{N-1}\#\mathcal{M}_j\le\sum_{j=0}^{N-1}\mu(\|\sigma-\sigma_j\|_{\A}^2+\osc^2(f,\mathcal{T}_j))^{-\frac1s}.$$
From the convergence result in Theorem \ref{thm:Convergence} we have
$\epsilon_N\le\alpha^{N-j}\epsilon_j$ for any $0\le j\le N-1$, which, along with the fact $\epsilon_j\approx \|\sigma-\sigma_j\|_{\A}^2+\osc^2(f,\mathcal{T}_j)$, implies
$$\#\mathcal{T}_N-\#\mathcal{T}_0\lesssim\mu(\|\sigma-\sigma_N\|_{\A}^2+\osc^2(f,\mathcal{T}_N))^{-\frac1s}\sum_{j=0}^{N-1}\alpha^\frac{j}{s}.$$
Since $\alpha<1$, the term $\sum_{j=0}^{N-1}\alpha^\frac{j}{s}$ is bounded. The definition of $\mu$ leads to
$$\|\sigma-\sigma_N\|_{\A}^2+\osc^2(f,\mathcal{T}_N)\lesssim |u,f|_{s}(\#\mathcal{T}_N-\#\mathcal{T}_0)^{-s},\quad \mbox{for~} (u,f)\in\mathbb{A}_s.$$
\end{proof}

\section{Applications}
In this section, we present two examples which satisfy these five hypotheses. The first example is  the mixed finite element of the Poisson equation; the
 second one is the mixed finite element of the Stokes problem within  the pseudostress-velocity formulation.  For both the 2D and 3D, we prove that
 the Raviart--Thomas and  Brezzi--Douglas--Marini elements  satisfy  Hypotheses \ref{H1}--\ref{lemma:PostreioriEstimation} in  Section 2.
 Hence the corresponding adaptive algorithms converge at the optimal rate in the nonlinear approximation sense.

In  the sequel, we use superscript $P$ to denote the  subspace or operator for the Poisson problem, and  $S$ to denote the subspace or operator for the Stokes problem.

\subsection{The Poisson problem}
The Raviart--Thomas element spaces \cite{brezzi1991mixed} are defined for  $k\ge 0$ by
\begin{equation*}
RT_h^P=\Sigma_{h,k}^P\times U_{h,k}^P,
\end{equation*}
where
\begin{equation*}
\Sigma_{h,k}^P:=\{\tau\in H(\ddiv,\Omega;\mathbb{R}^d):\tau |_K\in P_k(K)^d+x P_k(K), \forall K\in \mathcal{T}_h\},
\end{equation*}
and
\begin{equation*}
U_{h,k}^P:=\{v\in L^2(\Omega;\mathbb{R}): v|_K\in P_k(K), \forall K\in \mathcal{T}_h\}.
\end{equation*}
Here $P_k(K)$ denotes the space of polynomials of degree at most $k$ over $K$.

The Brezzi--Douglas--Marini element spaces \cite{brezzi1991mixed} are defined for  $k\ge 0$ by
\begin{equation*}
BDM_h^P=\Sigma_{h,k}^P\times U_{h,k}^P,
\end{equation*}
where
\begin{equation*}
\Sigma_{h,k}^P:=\{\tau\in H(\ddiv,\Omega;\mathbb{R}^d):\tau |_K\in P_{k+1}(K)^d, \forall K\in \mathcal{T}_h\},
\end{equation*}
and
\begin{equation*}
U_{h,k}^P:=\{v\in L^2(\Omega;\mathbb{R}): v|_K\in P_k(K), \forall K\in \mathcal{T}_h\}.
\end{equation*}

\begin{thm}\label{thm: Verification1}
For the Raviart--Thomas and  Brezzi--Douglas--Marini elements of the Poisson problem, Hypotheses \ref{H1}--\ref{H3} and \ref{lemma:PostreioriEstimation} hold.
\end{thm}
\begin{proof} Hypothesis \ref{H1} is an immediate result of the definitions  of finite element subspaces, while Hypotheses \ref{H2}, \ref{H3},  and \ref{lemma:PostreioriEstimation} were proved  in \cite[Lemma 2.1]{lovadina2006energy}, \cite[Lemma 2.8]{huang2012convergence}, \cite[Theorem 2.1, 2.2]{huang2012convergence},  respectively.
\end{proof}

In order to  check Hypothesis \ref{H4}, we  need  the following two spaces:
$$S_{h,k}:=\{\psi\in H^1(\Omega;\mathbb{R}): \psi|_K\in P_k(K), \forall K\in\mathcal{T}_h\},\quad d=2,$$
$$ND_{h,k}:=\{\psi\in H(\Curl,\Omega;\mathbb{R}^3): \psi|_K\in ND_k(K), \forall K\in\mathcal{T}_h\},\quad d=3,$$
where $ND_k(K):=P_{k-1}(K)^3\oplus\{v\in\tilde{P}_k(K)^3, v\cdot x=0,\forall x\in K\}$, with $\tilde{P}_k(K)$ being the spaces of homogeneous polynomials of  degree $k$  over $K$.

\begin{lem}\label{lemma:Zhong2011}
\cite[Theorem 4.1]{zhong2012convergence} \cite[Lemma 2.10]{huang2012convergence} There exists a quasi-interpolation operator $\mathcal{P}_{H,k}: ND_{h,k}\rightarrow ND_{H,k}$ such that for any $\varphi\in ND_{h,k}$,
\begin{equation}\label{eq: eq2}
\mathcal{P}_{H,k}\varphi|_{\mathcal{T}_H \setminus\tilde{\mathcal{R}}}=\varphi|_{\mathcal{T}_H \setminus\tilde{\mathcal{R}}},
\end{equation}
\begin{equation}\label{eq: eq3}
\|\Curl \mathcal{P}_{H,k}\varphi\|_{L^2(\Omega)}\lesssim  \|\Curl\varphi\|_{L^2(\Omega)}.
\end{equation}
\end{lem}

\begin{lem}\label{lemma:Schoberl2008}
\cite[Theorem 1]{schoberl2008posteriori} There exists an operator $\mathcal{I}_H: H(\Curl,\Omega;\mathbb{R}^3)\rightarrow ND_{H,1}$ with the following properties:
For every $\varphi\in H(\Curl,\Omega;\mathbb{R}^3),$ there exist $\psi\in H^1(\Omega;\mathbb{R}^3)$ and $w\in H^1(\Omega;\mathbb{R})$ such that
$$\varphi-\mathcal{I}_H\varphi=\psi+\grad w,$$
and
\begin{equation*}
\begin{split}
&h_K^{-1}\|w\|_{L^2(K)}+\|\grad w\|_{L^2(K)}\lesssim \|\varphi\|_{L^2(\Omega_K)},\\
& h_K^{-1}\|\psi\|_{L^2(K)}+\|\grad \psi\|_{L^2(K)}\lesssim \|\Curl\varphi\|_{L^2(\Omega_K)}.
 \end{split}
 \end{equation*}
\end{lem}

Now we are ready to present the following theorem.
\begin{thm}\label{thm:Interpolation2}
For the Raviart--Thomas and  Brezzi--Douglas--Marini elements of the Poisson problem,  Hypothesis \ref{H4}  holds.
\end{thm}
\begin{proof}  For the 2D case, $H(\Curl,\Omega)=H^1(\Omega;\mathbb{R})$.  Let $H_h^P(\Curl,\Omega)=S_{h,k}$,  and $\Pi_H^P$  be the Scott-Zhang interpolation operator \cite{scott1990finite}. Then \eqref{eq: CommutativeProperty} is true. Let $\psi=\varphi_h-\Pi_H^P\varphi_h$, and $\phi=0$,   which implies that the  decomposition   \eqref{eq: decompositon}, the estimates \eqref{eq: Interpolation2} hold. 

 For the  3D case,  let $H_h^P(\Curl,\Omega)=ND_{h,k}$,
and  $\Pi_H^P=\mathcal{I}_H+\mathcal{P}_{H,k}-\mathcal{I}_H\mathcal{P}_{H,k}$ where the operate $\mathcal{I}_H$ is from  Lemma \ref{lemma:Schoberl2008}, and the operator $\mathcal{P}_{H,k}$ is from Lemma \ref{lemma:Zhong2011}.
Then \eqref{eq: CommutativeProperty} is true. In addition, from Lemma \ref{lemma:Schoberl2008}, there exist
 $\psi\in H^1(\Omega;\mathbb{R}^3)$ and $w\in H^1(\Omega;\mathbb{R})$ such that
\begin{equation}\label{eq: PoissonInterpolation}
\varphi_h-\Pi_H^P\varphi_h=(I-\mathcal{I}_H)(I-\mathcal{P}_{H,k})\varphi_h=\psi+\grad w,
\end{equation}
and
\begin{equation}\label{eq: PoissonInterpolationProperty}
h_K^{-2}\|\psi\|^2_{L^2(K)}+\|\grad \psi\|^2_{L^2(K)}\lesssim \|\Curl(I-\mathcal{P}_{H,k})\varphi_h\|^2_{L^2(\Omega_K)}.
\end{equation}
From \eqref{eq: PoissonInterpolation}, \eqref{eq: PoissonInterpolationProperty}, the trace theorem, and Lemma \ref{lemma:Zhong2011},
we  obtain \eqref{eq: decompositon} and \eqref{eq: Interpolation2} with $\phi=0$.
\end{proof}

\subsection{The Stokes problem}
The finite element spaces for the Stokes problem are defined rowwise based on  those for  the Poisson problem case, with an additional restriction that the mean of the trace of the stress vanishes, namely,
$$
\Sigma_{h,k}^S:=\left\{\tau\in (\Sigma_{h,k}^P)^d |\int_\Omega\tr \tau dx=0\right\}, \text{ and } U_{h,k}^S:= (U_{h,k}^P)^d.
$$

\begin{thm}\label{thm: Verification3}
For the Raviart--Thomas and  Brezzi--Douglas--Marini elements of the Stokes problem, Hypotheses \ref{H1}--\ref{H3} and \ref{lemma:PostreioriEstimation} hold.
\end{thm}
\begin{proof}
The proofs of Hypothesis \ref{H1} and \ref{H3} are similar to those for  the Poisson problem. The proof of Hypothesis \ref{lemma:PostreioriEstimation} can be found in \cite[Theorem 5.4]{carstensen2011priori} for the 2D case while the proof  of  the 3D case is similar.
In order to show Hypothesis \ref{H2}, we  need to handle the additional restriction on  the mean of the trace of the stress.
In fact,  given $v_h=(v_1,\cdots,v_d)^T\in U_{h,k}^S$ with $v_i\in U_{h,k}^P$ $(i=1,\cdots,d)$, there exists $\tau_i\in\Sigma_{h,k}^P$ such that
$$\|v_i\|_{1,h}\lesssim \frac{(\ddiv\tau_i,v_i)_{L^2(\Omega)}}{\|\tau_i\|_{L^2(\Omega)}}.$$
Let $\tau_h=(\tau_1,\cdots,\tau_d)^T$, and  define
$$
\tilde{\tau}_h:=\tau_h-\frac{\int_\Omega \tr\tau_h dx}{d|\Omega|}I_{d\times d},
$$
then $\tilde{\tau}_h\in \Sigma_{h,k}^S$ and $$(\ddiv\tilde{\tau}_h,v_h)_{L^2(\Omega)}=(\ddiv\tau_h,v_h)_{L^2(\Omega)}.$$
Since $$ \|\tilde{\tau}_h\|^2_{L^2(\Omega)}=\|\tau_h\|^2_{L^2(\Omega)}-\frac{(\int_\Omega\tr\tau_hdx)^2}{d|\Omega|}\le\|\tau_h\|^2_{L^2(\Omega)},$$
 we immediately have $$\|v_h\|_{1,h}\lesssim \frac{(\ddiv\tau_h,v_h)_{L^2(\Omega)}}{\|\tau_h\|_{L^2(\Omega)}}\le \frac{(\ddiv\tilde{\tau}_h,v_h)_{L^2(\Omega)}}{\|\tilde{\tau}_h\|_{L^2(\Omega)}}.$$
 This proves Hypothesis \ref{H2}.
\end{proof}

\begin{thm}\label{thm:Interpolation4}
For  the Raviart--Thomas and Brezzi--Douglas--Marini elements of the Stokes problem, Hypothesis \ref{H4} holds.
\end{thm}

\begin{proof} For this case,  the finite element space $H_h^S(\Curl,\Omega):=(H_h^P(\Curl,\Omega))^d$.
 For $\varphi_h=(\varphi_1,\cdots, \varphi_d)^T\in H_h^S(\Curl,\Omega)$, define $\tilde{\Pi}_H^S\varphi_h=(\Pi_H^P\varphi_1,\cdots, \Pi_H^P\varphi_d)^T$, and
\begin{equation*}
\Pi_H^S\varphi_h:= \tilde{\Pi}_H^S\varphi_h-\frac{\int_\Omega \tr(\Curl\tilde{\Pi}_H^S\varphi_h)dx}{d|\Omega|}\left(\begin{array}{c}x_2\\-x_1\end{array}\right), \quad d=2,
\end{equation*}
$$
\Pi_H^S\varphi_h:= \tilde{\Pi}_H^S\varphi_h-\frac{\int_\Omega \tr(\Curl\tilde{\Pi}_H^S\varphi_h)dx}{2d|\Omega|}
\left(\begin{array}{ccc}0 &-x_3 &x_2\\x_3 &0 &-x_1\\-x_2 &x_1 &0\end{array}\right), \quad d=3.
$$
This yieds
$$
\Curl\Pi_H^S\varphi_h= \Curl\tilde{\Pi}_H^S\varphi_h-\frac{\int_\Omega \tr(\Curl\tilde{\Pi}_H^S\varphi_h)dx}{d|\Omega|}I_{d\times d} \mbox{~and~} \int_\Omega \tr(\Curl\Pi_H^S\varphi_h) dx=0,
$$
and so $\Curl\Pi_H^S\varphi_h\in\Sigma_H^S$.

For any $\varphi_i\in H_h^P(\Curl,\Omega)$, from Theorem \ref{thm:Interpolation2} there exists $\psi_i\in H^1(\Omega)$ such that
$$\Curl(\varphi_i-\Pi_H^P\varphi_i)=\Curl\psi_i,$$
and $\psi_i$ satisfy the condition of \eqref{eq: Interpolation2}.  Let $\psi:=(\psi_1,\cdots,\psi_d)^T$. It holds that
$$
\Curl(\varphi_h-\Pi_H^S\varphi_h)=\Curl\psi+\frac{\int_\Omega \tr(\Curl\tilde{\Pi}_H^S\varphi_h)dx}{d|\Omega|}I_{d\times d}.
$$
Hence the condition \eqref{eq: Interpolation2} is satisfied with  $\phi:=\frac{\int_\Omega \tr(\Curl\tilde{\Pi}_H^S\varphi_h)dx}{d|\Omega|}I_{d\times d}$. In addition, $(\A\sigma_H,\phi)_{L^2(\Omega)}=0$.

\end{proof}

\bibliographystyle{abbrv}
\bibliography{Hu-ref-2015-12-20}

\end{document}